\documentclass[11pt, a4paper]{article}
\usepackage{epsfig, graphicx, cite, amssymb, amsthm, color}
\usepackage[english]{babel}
\usepackage{amsmath,amsfonts,amscd, mathrsfs,enumerate}
\usepackage{geometry}

\usepackage{url}
\usepackage{hyperref}
\usepackage{authblk}

\usepackage{url,color,srcltx,mathrsfs}

\begin{document}

\def\K{\mathbb{K}}
\def\R{\mathbb{R}}
\def\C{\mathbb{C}}
\def\Z{\mathbb{Z}}
\def\Q{\mathbb{Q}}
\def\D{\mathbb{D}}
\def\N{\mathbb{N}}
\def\T{\mathbb{T}}
\def\P{\mathbb{P}}
\def\A{\mathscr{A}}
\def\CC{\mathscr{C}}
\def\supp{\mathrm{supp}}
\renewcommand{\theequation}{\thesection.\arabic{equation}}
\newtheorem{theo}{Theorem}[section]

\newtheorem{lemma}{Lemma}[section]
\newtheorem{coro}{Corollary}[section]
\newtheorem{prop}{Proposition}[section]
\newtheorem{definition}{Definition}[section]
\newtheorem{remark}{Remark}[section]

\newtheorem{example}{Example}[section]
\newtheorem{notation}{Notation}
\newtheorem{con}{Consequence}
\bibliographystyle{plain}
\theoremstyle{plain}

\title{\textbf {Solving the $\partial \overline{\partial}$ for extendable currents without vanishing the boundary cohomology group}}

\author{Mamadou Eramane Bodian$^{\ast}$ }
\affil{University Assane Seck of Ziguinchor, BP: 523 (Sénégal)}
\affil{m.bodian@uni-zig.sn$^{\ast}$}
\author{ Souhaibou Sambou }
\affil{University Gaston Berger of Saint-Louis (Sénégal)}
\affil{souhaibou.sambou@ugb.edu.sn }

\author{ S\'eny Diatta }
\affil{University Assane Seck of of Ziguinchor, BP: 523 (Sénégal)}
\affil{s.diatta1375@zig.univ.sn }

\author{ Salomon Sambou }
\affil{University Assane Seck of Ziguinchor, BP: 523 (Sénégal)}
\affil{ssambou@univ-zig.sn }
\date{}

\maketitle

\begin{abstract}~\\
In this paper, we consider the problem of solving the $\partial\overline{\partial}$ equation with discribed support for differential forms in a relatively compact domain $\Omega$ of a complex analytic manifold $X$. And as a consequence, we have the solution of the equation $\partial\overline{\partial}$ for extendable currents without the annulation assumption of the De Rham cohomology group of the boundary.   
\vskip 1.3mm
\noindent
{\bf Mathematics Subject Classification (2020,2010)} 32D15, 32D20, 31C10, 32L20 32F32.
\end{abstract}

\section*{Statements and Declarations}
\subsection*{Competing Interests}
 The authors attest that there are  non-financial interests that are directly or indirectly related to the work submitted for publication.
\subsection*{Data availability statements}
We have not included any data in this manuscript.
\subsection*{Ethics approval and consent to participate}
Ethical authors approved and consented to participate.
\subsection*{Consent for publication }
The authors have consented to the publication of the article.
\subsection*{Funding}
The authors have no funding.
\subsection*{Authors' contributions}
The authors have worked fairly in the development and work of the article.
\subsection*{Acknowledgements }
The authors have no acknowledgments.

\section*{Introduction}

 In this paper, we are interested in the $\partial \overline{\partial}$ problem with discribed support for differential forms and for currents. The problem is as follows,
let $X$ be an analytic complex manifold and $\Omega \subset X$ a relatively compact domain such that $X \setminus \Omega$ is connected. Then, under certain assumptions on the vanishing of De Rham and Dolbeault cohomology groups with discribed support, we have $$H^{1,1}_{BC, \overline{\Omega}}(X) = 0.$$ Our first result is the generalization of \cite[Theorem 1.2]{WG} to the case where $X$ is a Stein manifold and as an assumption about the vanishing of the $2$-th De Rham cohomology group with compact support see Theorem \ref{a}. Secondly, we give a more general result (see Theorem \ref{s}), where $X$ is an analytic complex  manifold assuming that the compact De Rham and Dolbeault cohomology groups $$H^2_{c}(X) = H^{0,1}_{c}(X) = 0.$$ The corollary \ref{b} is a solution of the $\partial \overline{\partial}$ for $(1,1)$-differential forms in a completely pseudoconvex smooth-boundary domain with discribed support. And by duality, we have the solution of $\partial \overline{\partial}$ for $(n-1,n-1)$-extendable currents without the vanishing hypothesis on the De Rham cohomology group of the boundary see corollary \ref{c}. Finally, we etablish in theorem \ref{Era} in the last section, the extension of these results to differential (p,q)-forms under certain conditions and as a consequence we have by duality the solution of the $\partial \overline{\partial}$-problem for $(n-p,n-q)$-extendable currents see corollary \ref{M}.

In case of solving the $\partial
\bar{\partial}$ equations for extendable currents, vanishing the De Rahm cohomology group of the boundary was a necessary assumption for the
solution with exact support. See for exemple \cite{SBD, SBIH}.    However this present result removes this
ambiguity.

\section{Notations}
We fix the following notations:
Let $X$ be a complex manifold
\begin{itemize}
\item[•] $\displaystyle{ \mathcal{E}^k (X) :=}$  the space of $k$-form of class $C^{\infty}$ on $X$;
\item[•] $\displaystyle{\mathcal{D}^k(X) :=}$  the space of $k$-form of class $C^\infty$ with compact support;
\item[•] $\displaystyle{\mathcal{D}^{'k}(X) :=}$ the space of $k$-current of $X$.
\end{itemize}

For every $k \in \{0,1,\ldots,2n\}$, we have
 \[ \displaystyle{\mathcal{E}^k(X) = \oplus_{p+q=k} \mathcal{E}^{p,q}(X)} \]
and the complex structure of $X$ induices a splitting 
\[ d = \partial + \overline{\partial}.\]
For every $p,q \in \{1, \ldots,n \}$, we define the Bott-Chern cohomology group of smooth $(p,q)$-forms on $X$ as
\[H_{BC}^{p,q}(X) = \frac{\left(\ker \partial : \mathcal{E}^{p,q}(X) \rightarrow \mathcal{E}^{p+1,q}(X) \right) \cap \left( \ker \overline{\partial} : \mathcal{E}^{p,q}(X) \rightarrow \mathcal{E}^{p,q+1}(X) \right)}{Im \partial \overline{\partial}: \mathcal{E}^{p-1,q-1}(X) \rightarrow \mathcal{E}^{p,q}(X).  }\]
 $\displaystyle{H_{BC,\overline{\Omega}}^{p,q}(X)}$ is the Bott-Chern cohomology group of smooth $(p,q)$-forms with support in $\overline{\Omega}.$\\
We define the Applie cohomology group of smooth $(p,q)$-forms on $X$ as 

\[H_{A}^{p,q}(X) = \frac{\ker (\partial \overline{\partial }) : \mathcal{E}^{p,q}(X) \rightarrow \mathcal{E}^{p+1,q+1}(X)}{Im \partial: \mathcal{E}^{p-1,q}(X) \rightarrow \mathcal{E}^{p,q}(X) + Im \overline{\partial}:  \mathcal{E}^{p,q-1}(X) \rightarrow \mathcal{E}^{p,q}(X).  }\]

Case of either $p=0$ or$q=0$. For example, if $q=0$, then the $(p,0)$ Bott-Chern cohomology group is given, from definition, by
\[H_{BC}^{p,0}(X) = \{ f \in \Gamma(X, \Omega_{X}^{p}), | ~\partial f :=0 \}, \]
where $\Omega_{X}^{p}$ is the sheaf of holomorphic $p$-forms on $X.$ Thanks to the symmetric property of Bott-Chern cohomology group, we have  
\[H_{BC}^{0,p}(X) =  \overline{H_{BC}^{p,0}(X) }. \]

$H_{A}^{0,0}(X)$ is the set of pluriharmonic functions on $X.$

The Bott-Chern cohomology group of $(p,q)$-forms with compact support is defined as

\[H_{BC,c}^{p,q}(X) = \frac{\left(\ker \partial : \mathcal{D}^{p,q}(X) \rightarrow \mathcal{D}^{p+1,q}(X) \right) \cap \left( \ker \overline{\partial} : \mathcal{D}^{p,q}(X) \rightarrow \mathcal{D}^{p,q+1}(X) \right)}{Im \partial \overline{\partial}: \mathcal{D}^{p-1,q-1}(X) \rightarrow \mathcal{D}^{p,q}(X).  }\]

\section{Solving the $\partial \overline{\partial}$ for extendable currents}

The following result generalizes \cite[Theorem 1.2]{WG},
\begin{theo} \label{a}~\\

Let $X$ be a Stein manifold of complex dimension $n \geq 2$ such that $$H^2_{c}(X) = 0.$$ Let $\Omega \subset X$ be a relatively compact domain such that $X \setminus \Omega$ is connected. 
Then $$H^{1,1}_{BC, \overline{\Omega}}(X) = 0.$$

\end{theo}
\begin{proof}~\\

Let $f$ be a $(1,1)$-differential form of class $C^\infty$, $d$-closed with support on $\overline{\Omega}$. Since $H^2_{c}(X) =0$, there exists a $1$-differential form $g$ of class $C^\infty$ with compact support such that $dg = f$.  
We have $g= g_1+ g_2$ where $g_1$ and $g_2$ are respectively a $(0,1)$ differential form of class $C^\infty$ with compact support and a $(1,0)$ differential form of class $C^\infty$ with compact support. We have $dg= dg_1 +dg_2=f$.
Since $d= \partial + \overline{\partial}$ and for bidegree reasons we have $\partial g_2 =0$ and $\overline{\partial} g_1 =0$. Since $X$ is Stein, there exist functions $w_1$ and $w_2$ of class $C^\infty$ with compact support such that $g_1 = \overline{\partial}w_1$ and $g_2= \partial w_2$.
\begin{center}
$f = \partial g_1+ \overline{\partial}g_2= \partial \overline{\partial}w_1 +\overline{\partial}\partial w_2 = \partial \overline{\partial}(w_1 - w_2)$.
\end{center}

We set $v= w_1 - w_2$, then $v$ is compactly supported and $\partial \overline{\partial}v =0$ on $X \setminus \overline{\Omega}$. According to \cite[theorem 1.1]{WG}, there exists a pluriharmonic function $\tilde{v}$ defined on $X$ such that $\tilde{v}_{X \setminus \overline{\Omega}}= v$. We set $$\check{v}= v - \tilde{v}.$$
Then we have $supp( \check{v}) \subset \overline{\Omega}$ and $\partial \overline{\partial} \check{v}= f$. So $$H^{1,1}_{BC, \overline{\Omega}}(X) = 0.$$
\end{proof}
More generally, we have the following result
\begin{theo}\label{s}~\\

Let $X$ be an  analytic complex manifold of complex dimension $n \geq 2$ such that $$H^2_{c}(X) = H^{0,1}_{c}(X) = 0,$$ and $\Omega \subset X$ be a relatively compact domain such that $X \setminus \Omega$ is connected. 

Then, $$H^{1,1}_{BC, \overline{\Omega}}(X) = 0.$$

\end{theo}
\begin{proof}~\\

Let $f$ be a $(1,1)$-differential form of class $C^\infty$, $d$-closed with support on $\overline{\Omega}$. Since $H^2_{c}(X) =0$, there exists a $1$-differential form $g$ of class $C^\infty$ with compact support such that $dg = f.$  

We have $g= g_1+ g_2$ where $g_1$ and $g_2$ are respectively a $(0,1)$ differential form of class $C^\infty$ with compact support and a $(1,0)$ differential form of class $C^\infty$ with compact support. We have $dg= dg_1 +dg_2=f$.
Since $d= \partial + \overline{\partial}$ and for bidegree reasons we have $\partial g_2 =0$ and $\overline{\partial} g_1 =0$. Since $H^{0,1}_{c}(X) = 0$, there exist functions $w_1$ and $w_2$ of class $C^\infty$ with compact support such that $g_1 = \overline{\partial}w_1$ and $g_2= \partial w_2$.
\begin{center}
$f = \partial g_1+ \overline{\partial}g_2= \partial \overline{\partial}w_1 + \overline{\partial}\partial w_2 = \partial \overline{\partial}(w_1 - w_2).$
\end{center}
Let $v= w_1 - w_2$, $v$ is compactly supported and $\partial \overline{\partial}v =0$ on $X \setminus \overline{\Omega}$. Then $v$ is a pluriharmonic function on $X \setminus \overline{\Omega}$ and which cancels on $X \setminus \supp v$. Suppose that $\overline{\Omega}$ is strictly included in $\supp v$. So $v = 0$ on $X \setminus \overline{\Omega}$ which is absurd. So $supp v \subset \overline{\Omega}.$ This implies that $$H^{1,1}_{BC, \overline{\Omega}}(X) = 0.$$
\end{proof}

\begin{coro}\label{b}~\\

Let $X$ be a complex analytic manifold of complex dimension $n \geq 2$ and $\Omega \subset X$ a completely pseudoconvex domain with smooth boundary. Then $$H^{1,1}_{BC, \overline{\Omega}}(X) = 0.$$

\end{coro}
\begin{proof}~\\

Since $\Omega$ has a smooth boundary, according to \cite[Theorem 2.3]{SS}, there exists a completely strictly pseudoconvex domain $\Omega'$ such that $\Omega \subset \Omega' \subset \subset X$ and that $H^r(\Omega')= H^r(\Omega') =0$ for $r \leq 2$. By Poincaré duality $H^{n-r}_c(\Omega')=0$ $\Longleftrightarrow$ $H^{j}_c(\Omega')=0$ for $0 \leq j \leq n-2$. In particular $H^{2}_c(\Omega')=0$ and $\Omega'$ is also a Stein domain. According to theorem \ref{a}, $$H^{1,1}_{BC, \overline{\Omega}}(\Omega') = H^{1,1}_{BC, \overline{\Omega}}(X) = 0.$$
\end{proof}

\begin{remark}
Corollary \ref{b} shows that if $f$ is a $(1.1)-$differential form of class $C^\infty$ with support in $\overline{\Omega}$ and $d$-closed, then there exists a function $u$ of class $C^\infty$ with support in $\overline{\Omega}$ such that $\partial \overline{\partial} u= f.$
\end{remark}
By duality, we have the following result.
\begin{coro}\label{c}~\\
Under the assumptions of the corollary \ref{b}, if $T$ is a $(n-1,n-1)$-extendable current, $d$-closed and defined on $\Omega$, then there exists a $(n-2,n-2)$-extendable current $S$ such that $\partial \overline{\partial} S= T.$
\end{coro}
\begin{proof}~\\
We have  $\check{D}^{n-1,n-1} ( \Omega) = ( D^{1,1} ( \overline{\Omega}))'$. Let us show that  $$\check{D}^{n-1,n-1} (\Omega) \cap \ker d = \partial \overline{\partial} \check{D}^{n-2,n-2} ( \Omega).$$
  
\noindent  Let $T \in \check{D}^{n-1,n-1} ( \Omega) \cap \ker d$. We define
\[
\begin{split}
L_T : \partial \overline{\partial} D^{1,1} ( \overline{\Omega}) &\longrightarrow \mathbb{C}\\
 \partial \overline{\partial} \varphi&\longmapsto 
  \langle T , \varphi \rangle,
\end{split} 
\]
\noindent  $L_T$ is well defined because $\partial \overline{\partial} D^{n-2,n-2} ( \overline{\Omega}) = D^{n-1,n-1} (
  \overline{\Omega}) \cap \ker d$ (cf corollaire \ref{b}).

\noindent  If $\varphi$ and $\varphi'$ are two $(1,1)$-differential forms
  such that $\partial \overline{\partial} \varphi = \partial \overline{\partial} \varphi'$, then
  \[ \varphi - \varphi' = \partial \overline{\partial} \theta \quad \rm{et} \quad \langle T, \partial \overline{\partial} \theta
     \rangle = \lim_{j \rightarrow + \infty} \langle T, \partial \overline{\partial} \theta_j \rangle = 0
     \text{} \Rightarrow \langle T, \varphi \rangle = \langle T, \varphi'
     \rangle  \]
    where $ (\theta_{j})_{j\in \mathbb{N}} $ is a sequence of elements of $ D^{2,2}(\Omega) $ which converge uniformly towards $ \theta $. \\
\noindent  $L_T $ is linear because $\partial \overline{\partial}$ : $D^{2,2} ( \overline{\Omega}) \longrightarrow
  \partial \overline{\partial} D^{2,2} ( \overline{\Omega})$ is a linear and continuous map because the inverse image of an open set $ U $ of $\mathbb{C}$ by $L_T$ is an open set. So, we have $L_T \circ \partial \overline{\partial} = T$ since $L_T^{- 1} ( U) = \partial \overline{\partial} \circ T^{- 1} ( U)$. According to the Hahn-Banach theorem, we can extend   $L_T$ into a linear and continuous operator $\tilde{L}_T : D^{1,1} ( \overline{\Omega}) \longrightarrow \mathbb{C}$, it is an extendable current and
   $\partial \overline{\partial} \tilde{L}_T \varphi =  T$
  then $$\langle \partial \overline{\partial} \tilde{L}_T, \varphi \rangle =  \langle T, \varphi \rangle.$$ Set $S =
   \tilde{L}_T $, $S$ is an extendable current and verifies the relationship $\partial \overline{\partial} S = T$.

\end{proof}
\begin{remark}
The corollary \ref{c} is an improvement of \cite[Theorem 1.1]{BDS} without the assumption $H^j(b \Omega) = 0$ for all $1 \leq j < 2n - 1$.
\end{remark}
\section{The case of differential (p,q)-forms and extendable currents }
In this section we generalize these results to differential $(p,q)$-forms.Then we have :
\begin{theo}\label{Era} ~\\
Let $X$ be a complex analytic manifold of dimension $n$ and $\Omega
  \subset X$ a smooth boundary  domain and $1 \leq p,q \leq n$ such that $H^{p + q}_{c,
  d} (X) = H_{c, \partial}^{p - 1, q} (X) = H_{c,
  \overline{\partial}}^{p, q - 1} (X) = 0$. If the natural mapp
  \begin{eqnarray*}
    H^{p - 1, q - 1}_{A} (X) & \longrightarrow & H^{p - 1, q -
    1}_{A} (X \backslash \Omega)
  \end{eqnarray*}
  is an isomorphism then $H_{BC, \overline{\Omega}}^{p, q} (X) = 0.$
\end{theo}

\begin{proof}~\\
  Let $[f] \in H^{p, q}_{BC, \overline{\Omega}} (X) \Longrightarrow f$ is
  a differential $(p + q)$-form $d$-closed with compact support.
Since $H^{p + q}_{c, d} (X) = 0$, then there exists $g$ a differential $(p + q -
  1)$-form with compact support such that $d g = f$. We
  can decompose  $g$ into the sum of a  differential $(p - 1, q)$-form $g_1$ and a differential $(p, q -
  1)$-form $g_2$; i.e $g = g_1 + g_2$. The saturation of the bidegree holds  $\overline{\partial} g_1 = 0$ and
  $\partial g_2 = 0$ so there are two differential forms $h_1$ and $h_2$ with
  compact support such that $g_1 = \overline{\partial} h_1$ and $g_2 = \partial h_2.$
    Hence
  \begin{eqnarray*}
    f & = & \partial g_1 + \overline{\partial} g_2\\
    & = & \partial \overline{\partial} h_1 + \overline{\partial} \partial h_2\\
    & = & \partial \overline{\partial} (h_1 - h_2) .
  \end{eqnarray*}
  The solution $h_1 -h_2$ we get has compact support. We end the proof by correcting this solution to new one with exact support. 
\item[•] If $p=q=1$, then $h_1-h_2$ is a pluriharmonic function defined in $X \setminus \overline{\Omega}$. Under the theorem's assumptions, we have a natural isomorphism map 
  \begin{eqnarray*}
    H^{0, 0}_{A} (X) & \longrightarrow & H^{0, 0}_{A} (X \backslash \Omega).
  \end{eqnarray*}
   Then, there is $u$ a pluriharmonic function on $X$
  such that
  
  $[u] = [(h_1 - h_2)_{| X \backslash \overline{\Omega} }]
  \Longrightarrow u_{| X \backslash \overline{\Omega}} = (h_1 - h_2)_{| X \backslash \overline{\Omega} }$. Set $\widetilde{h_1 - h_2} = h_1 - h_2
  - u$ which is compactly supported on
  $\Omega$ and $\partial \overline{\partial} (\widetilde{h_1 - h_2}) = f
  \textnormal{ on } \Omega .$
  
  \item[•] If $p,q \geq 2,$ $h_1-h_2$ is a $(p-1,q-1)$ form with compact support.
  
 We have, $\partial \overline{\partial} (h_1 - h_2) = 0$ on $X \backslash
  \overline{\Omega}.$ Under the theorem's assumptions, we have a natural isomorphism map 
  \begin{eqnarray*}
    H^{p - 1, q - 1}_{A} (X) & \longrightarrow & H^{p - 1, q -
    1}_{A} (X \backslash \Omega).
  \end{eqnarray*}
   Then, there is $u$ a $(p-1,q-1)$ pluriharmonic form on $X$
  such that
  
  $[u] = [(h_1 - h_2)_{| X \backslash \overline{\Omega} }]
  \Longrightarrow u_{| X \backslash \overline{\Omega}} = (h_1 - h_2)_{| X \backslash \overline{\Omega} } +
  \partial \overline{\partial} v$, where $v$ is a $(p-2,q-2)$ form defined in $X\setminus \overline{\Omega}   .$ 
  Hence $(h_1 - h_2)_{| X \backslash
  \overline{\Omega} } = u - \partial \overline{\partial} \tilde{v}$ where
  $\tilde{v}$ is an extension of $v$. Set $\widetilde{h_1 - h_2} = h_1 - h_2
  - u + \partial \tilde{\partial} \tilde{v}$ which is compactly supported on
  $\Omega$ and $\partial \overline{\partial} (\widetilde{h_1 - h_2}) = f
  \textnormal{ on } \Omega .$
  
  \ 
\end{proof}
As a consequence, we have :
\begin{coro}\label{M}
Let $X$ be a complex analytic manifold of complex dimension $n \geq 2$ and $\Omega \subset X$ a completely pseudoconvex domain with smooth boundary. If $T$ is a $(n-p,n-q)$-extendable current, $d$-closed and defined on $\Omega$, then there exists a $(n-p-1,n-q-1)$-extendable current $S$ such that $\partial \overline{\partial} S= T$.
\end{coro}
\begin{proof}
The proof is identical to that of the corollary \ref{c}.
\end{proof}

\section{Conclusion}
In this paper, we solve the $\partial\bar{\partial}$-problem by a direct method based on Bott-Chern cohomology. What is original in this case is that, unlike in papers \cite{BDS},\cite{SS},\cite{SBD} and \cite{SBIH}, we did not use the null hypothesis of the boundary cohomology group of the domain, which was a necessary condition for the solution with discribed support.

\end{document}